\documentclass[11pt]{amsproc}
\usepackage{mathrsfs}
\usepackage{stmaryrd}
\usepackage{cases}
\usepackage{amssymb}
\usepackage{amsmath}
\usepackage{amsfonts}
\usepackage{graphicx}
\usepackage{amsmath,amstext,amsbsy,amssymb, color}
\newtheorem{theorem}{Theorem}[section]

\newtheorem{proposition}[theorem]{Proposition}

\theoremstyle{definition}
\newtheorem{definition}[theorem]{Definition}

\theoremstyle{remark}
\newtheorem{remark}[theorem]{Remark}

\numberwithin{equation}{section} \errorcontextlines=0

\newcommand{\Pf}{\mathrm{Pf}}

\newcommand{\cdet}{\mathrm{cdet}}
\newcommand{\rdet}{\mathrm{rdet}}

\newcommand{\ot}{\otimes}
\newcommand{\id}{\mathrm{id}}
\newcommand{\si}{\sigma}

\begin{document}

\title[Multiparameter quantum Pfaffians]
{Multiparameter quantum Pfaffians}
\author{Naihuan Jing}
\address{NJ: Department of Mathematics, Shanghai University,
Shanghai 200444, China and Department of Mathematics, North Carolina State University, Raleigh, NC 27695, USA}
\email{jing@math.ncsu.edu}
\author{Jian Zhang}
\address{JZ: Institute of Mathematics and Statistics, University of Sao Paulo, Sao Paulo, Brazil 05315-970}
\email{zhang@ime.usp.br}
\thanks{{\scriptsize
\hskip -0.6 true cm MSC (2010): Primary: 17B37; Secondary: 58A17, 15A75, 15B33, 15A15.
\newline Keywords: multiparameter quantum groups, $q$-determinants, $q$-Pfaffians, $q$-hyper-Pfaffians\\
Supported by NSFC (11531004), 
Fapesp (2015/05927-0) and Humboldt Foundation.}}

\begin{abstract}
The multiparameter quantum Pfaffian of the $(p, \lambda)$-quantum group is introduced and studied together with the quantum determinant,
and an identity relating the two invariants is given. Generalization to the multiparameter hyper-Pfaffian and relationship with
the quantum minors are also considered.
\end{abstract}
\maketitle
\section{Introduction}

In the early study of quantum groups, multiparameter quantum groups and quantum enveloping algebras
were considered along the line of one-parameter quantum groups \cite{T, S, R}. It was clear that many of their properties
are similar to their one-parameter analog, for example, Artin-Schelter-Tate \cite{AST} showed that the multiparameter
general linear quantum group has the same Hilbert function of the polynomial functions in $n^2$ variables under the
so-called $(p, \lambda)$-condition (see \eqref{e:cond-p}).
Further results have been established for two- and multi-parameter quantum groups \cite{Sc, BG, HLM, LS, JL}
such as the existence of the quantum determinant, which helps to transform quantum semigroups
into quantum groups.
Recently it is known that the quantum Pfaffians
can be extended to two-parameter quantum groups as well \cite{JZ2}.

In this paper, we generalize our recent study of quantum determinants and Pfaffians from
two-parameter quantum groups to multiparameter cases. We will adopt the same approach of quadratic algebras \cite{M} to study quantum
determinants and quantum Pfaffians, and express them as the scaling constants of
quantum differential forms (cf. \cite{JZ1}).
In particular, we will
prove that the multiparameter quantum Pfaffian can be defined for a more general class of multiparameter quantum matrices
and prove the identity between the quantum determinant and quantum Pfaffian, and also establish
their integrality property for quantum groups under the $(p, \lambda)$-conditions.

We also formulate the multiparameter quantum determinants in terms of the quasideterminant of the generating matrix.
Generalizing the one-form and two-form, we obtain higher degree analogs of the multiparameter Pfaffians and
establish the transformation rule of the multiparameter quantum hyper-Pfaffian under the quantum determinant,
which extends some of the transformation rules of Pfaffians in linear algebra.

\section{Quantum determinants}

\subsection{Quantum semigroup $\mathcal A$}

Let $p=(p_{ij}), q=(q_{ij})$ be two sets of $n^2$ parameters in the complex field $\mathbb C$ arranged in matrix forms satisfying the following relations:
\begin{align*}
p_{ij}p_{ji}=1,\quad p_{ii}=1;\quad q_{ij}q_{ji}=1,\quad q_{ii}=1.
\end{align*}
For a scalar $v$, the $v$-commutator $[x, y]_v$ is defined by
\begin{equation}
[x, y]_v=xy-vyx.
\end{equation}
Therefore two elements $x$ and $y$ are $q$-commutative if $[x, y]_{q_{ij}}=0$ 
for a parameter $q_{ij}$.

We define the unital algebra $\mathcal A$ as an associative complex algebra generated
by $a_{ij}$, $1\leqslant i, j\leqslant n$ subject to the following relations:
\begin{align}\label{relation1}
&a_{ik}a_{il}=q_{kl}a_{il}a_{ik}, \\ \label{relation2}
&a_{ik}a_{jk}=p_{ij}a_{jk}a_{ik}, \\ \label{relation3}
&a_{ik}a_{jl}-\frac{q_{kl}}{q_{ij}}a_{jl}a_{ik}=q_{kl}a_{il}a_{jk}-\frac{1}{q_{ij}}a_{jk}a_{il},\\ \label{relation4}
&a_{ik}a_{jl}-\frac{p_{ij}}{p_{kl}}a_{jl}a_{ik}=p_{ij}a_{jk}a_{il}-\frac{1}{p_{kl}}a_{il}a_{jk},
\end{align}
where $i<j$ and $k<l$. These can be paraphrased as that the
quantum matrix $A=(a_{ij})$ is row $q$-commutative, column $p$-commutative, and satisfies the equality between
the $q$-commutator (resp. $p$-commutator) of the diagonal entries and the $q$-commutator (resp. $p$-commutator)
of the skew diagonal entries up to a $q$-factor (resp. $p$-factor).

The algebra $\mathcal A$ has a bialgebra structure under the comultiplication
$\mathcal A\longrightarrow \mathcal A\ot\mathcal A$ given by
\begin{equation}
\Delta(a_{ij})=\sum_ka_{ik}\otimes a_{kj},
\end{equation}
and the counit given by $\varepsilon(a_{ij})=\delta_{ij}$, the Kronecker symbol.

For any permutation $\sigma$ in $S_n$, the $q$-inversion associated to the parameters $q_{ij}$ is defined as
\begin{equation}\label{e:qinv}
(-q)_{\sigma}=(-1)^{l(\sigma)}\prod_{\substack{i<j\\\si_i>\si_j}}q_{\sigma_{j}\sigma_{i}},
\end{equation}
where $l(\sigma)=|\{(i, j)|i<j, \si_i>\si_j\}|$ is the classical inversion number of $\sigma$.

The $v$-based quantum number is defined to be:
\begin{equation}
[n]_v=1+v+\cdots+v^{n-1},
\end{equation}
and the quantum factorial $[n]_v!=[1]_v[2]_v\cdots [n]_v$.

Let $A=(a_{ij})$ be the matrix with entries satisfying \eqref{relation1}-\eqref{relation4}.
We define the quantum row-determinant and column-determinant of $A$ as follows.
\begin{align}\label{e:qdet}
\rdet(A)&=\sum_{\sigma\in S_n}(-q)_{\sigma}a_{1,\sigma_1}\cdots a_{n,\sigma_n},\\ \label{e:qper}
\cdet(A)&=\sum_{\sigma\in S_n}(-p)_{\sigma}a_{\sigma_1,1}\cdots a_{\sigma_n,n}.
\end{align}
 The first property we show is that both are group-like elements:
\begin{align*}
\Delta(\rdet(A))&=\rdet(A)\otimes \rdet(A),\\
\Delta(\cdet(A))&=\cdet(A)\otimes \cdet(A).
\end{align*}

To do this we
introduce two copies of commuting
quantum exterior algebras associated to the parameters $p_{ij}$ and $q_{ij}$ respectively. The first one is
$$\Lambda_q(x)=\mathbb C\langle x_1, \ldots,x_n\rangle /I$$
where $I$ is the ideal
$\left( x_i^2, q_{ij}x_ix_j+x_jx_i|1\leqslant i<j\leqslant n\right)$ and one writes $x\wedge x'=x\otimes x'\mod I$. Then we have that
\begin{align}\label{e:wedge1}
&x_j \wedge x_i=-q_{ij}x_i\wedge x_j, \\ \label{e:wedge2}
&x_i\wedge x_i=0,
\end{align}
where $i<j$. Then for $\sigma\in S_n$,
\begin{equation}
x_{\sigma_1}\wedge\cdots\wedge x_{\sigma_n}=(-q)_{\si}x_1\wedge\cdots\wedge x_{n}.
\end{equation}
Clearly $\Lambda_q$ is a left $\mathcal A$-comodule with the coaction
$\mu_q:\Lambda_q\rightarrow \mathcal A\ot\Lambda_q$ given by
\begin{equation}
\mu_q(x_{i})=\sum a_{ij}\ot x_{j}.
\end{equation}

 The row determinant can be computed via the coaction:
\begin{equation}
\mu_q(x_1\wedge\cdots\wedge x_n)=\rdet(A)\otimes x_1\wedge\cdots \wedge x_n.
\end{equation}
Subsequently the comodule structure map
$(\id\ot\mu_q)\mu_q=(\Delta\ot\id)\mu_q$ implies that $\Delta(\rdet(A))=\rdet(A)\otimes \rdet(A)$.

Let $\Lambda_p=\Lambda_p(y)$ be the unital associative algebra
$\mathbb C\langle y_1, \ldots,y_n\rangle/J$, where $J$ is the ideal
$(y_i^2, p_{ij}y_iy_j+y_jy_i | 1\leqslant i<j\leqslant n)$. Using similar convention
for $x_i$'s, the relations are
\begin{align}\label{qwedge3}
&y_j \wedge y_i=-p_{ij}y_i\wedge y_j, \\
&y_i\wedge y_i=0,
\end{align}
where $1\leqslant i<j\leqslant n$. The space $\Lambda_p$ is a right $\mathcal A$-comodule with coaction
$
\mu_p':\Lambda_p\rightarrow  \Delta_p\ot \mathcal A
$
given by
\begin{equation}\mu'_p(y_{i})=\sum_{j=1}^n y_{j}\ot a_{ji}.
\end{equation}
Similar to $\rdet$ one has that
$\Delta(\cdet(A))=\cdet(A)\otimes \cdet(A)$. In general, $\rdet\neq \cdet$. However we will consider a
special case while the two determinants are equal.

From now on we assume that $(p_{ij}, q_{ij})$ live on the parabola
\begin{equation}\label{e:cond-p}
p_{ij}q_{ij}=\lambda, \quad \lambda\neq -1, \quad i<j.
\end{equation}
These relations \eqref{e:cond-p} are referred as {\it the $(p, \lambda)$ conditions} \cite{AST}.
In the following whenever we talk about
the multiparameter quantum groups we always consider those satisfying the $(p, \lambda)$ relations.

\begin{theorem}
In the bialgebra $\mathcal A$ with the $(p, \lambda)$ relations, one has that
$$\rdet (A)=\cdet (A).$$
\end{theorem}
\begin{proof}
Consider the following special linear element $\Phi$ in $\mathcal A\otimes \Lambda_q\otimes \Lambda_p$:
\begin{equation}
\Phi=\sum_{i,j=1}^n a_{ij}\ot y_i\ot x_j=y^TAx,
\end{equation}
where we have set $x=(x_1, \ldots, x_n)^T$ and $y=(y_1, \ldots, y_n)^T$.

Put further $\delta=(\delta_1, \ldots, \delta_n)^T$,
$\partial=(\partial_1, \ldots, \partial_n)^T$, and consider the following elements:
\begin{align}
\delta &=Ax, \\
\partial &=A^Ty.
\end{align}

Let $\omega_i=y_i\delta_i=\sum_{j=1}^n y_ia_{ij}x_j$. It follows from the relations
\eqref{relation1}-\eqref{relation4} and the commutation relations of $\Lambda_q$ and $\Lambda_p$ that
\begin{align}\label{e:ext1}
&\omega_i \wedge \omega_i=0, \quad 1\leqslant i\leqslant n, \\ \label{e:ext2}
&\omega_j \wedge \omega_i=\lambda \omega_i \wedge w_j, \quad 1\leqslant i<j\leqslant n.
\end{align}

It follows from (\ref{e:ext1})-(\ref{e:ext2}) that
\begin{align*}
\wedge^n\Phi &=(\sum_{\sigma\in S_n}\lambda^{l(\sigma)})\omega_1\wedge\cdots\wedge\omega_n \\
&=[n]_{\lambda}!(x_1\wedge\cdots\wedge x_n)(\partial_1\wedge\cdots\wedge\partial_n)\\
&=[n]_{\lambda}!\rdet(A) (  y_1\wedge\cdots\wedge y_n) (x_1\wedge\cdots\wedge x_n).
\end{align*}

Note that one can also rewrite $\Phi=\sum_{i=1}^n\omega_i'$ with $\omega_i'=\partial_ix_i=\sum_{j=1}^n a_{ji}y_jx_i$.
Then it follows that
\begin{align*}
\wedge^n\Phi=[n]_{\lambda}!\cdet(A) (  y_1\wedge\cdots\wedge y_n) (x_1\wedge\cdots\wedge x_n),
\end{align*}
which implies that
$$\rdet(A)=\cdet(A).$$
\end{proof}

Due to this identity, from now on, we will define the multiparameter quantum determinant for the
$(p, \lambda)$-quantum group as
\begin{equation}\label{e:qdet2}
\begin{aligned}
{\det}_q(A)&=\sum_{\sigma\in S_n}(-q)_{\sigma}a_{1,\sigma_1}\cdots a_{n,\sigma_n}\\
&=\sum_{\sigma\in S_n}(-p)_{\sigma}a_{\sigma_1,1}\cdots a_{\sigma_n,n}.
\end{aligned}
\end{equation}

For a pair of $t$ indices $i_1, \ldots, i_t$ and $j_1, \ldots, j_t$,
we define the quantum row-minor $\det_q(A_{j_1 \ldots j_t}^{i_1 \ldots i_t})$
as in (\ref{e:qdet2}).
Like the determinant, the quantum row minor row also equals to the quantum column minor
for any pairs of ordered indices $1\leqslant i_1<\cdots<i_t\leqslant n$ and $1\leqslant j_1<\cdots<j_t
\leqslant n$, which justifies the notation.

For any $t$ indices $i_1, \ldots, i_t$ 
\begin{align}\label{e:minor1}
\delta_{i_1}\wedge\cdots\wedge\delta_{i_t}=\sum_{j_1<\cdots<j_t}{\det}_q(A^{i_1\ldots i_t}_{j_1\ldots j_t})x_{j_1}\wedge\cdots\wedge x_{j_t},
\end{align}
where the sum runs through all indices $1\leqslant j_1<\cdots<j_t\leqslant n$. This
implies that ${\det}_q(A^{i_1\ldots i_t}_{j_1\ldots j_t})=0$ whenever there are two identical rows.

As $\delta_i$'s obey the wedge relations \eqref{e:wedge1}-\eqref{e:wedge2}, for any $t$-shuffle $\sigma\in S_n$ : $1\leqslant \sigma_1<\cdots<\sigma_t, \sigma_{t+1}<\cdots<\sigma_n\leqslant n$,
one has that
$$\delta_{\sigma_1}\wedge\cdots\wedge\delta_{\sigma_t}\wedge \delta_{\sigma_{t+1}}\wedge\cdots\wedge\delta_{\sigma_n}
=(-q)_{\sigma}\delta_1\wedge\cdots\wedge\delta_n.
$$
Note that $x_j$'s also satisfy the same wedge relations. This
then implies the following Laplace expansion by invoking \eqref{e:minor1}: for each fixed $t$-shuffle
$\sigma_1<\cdots<\sigma_t, \sigma_{t+1}<\cdots<\sigma_n$, one has that
\begin{equation}\label{e:lap1}
{\det}_q(A)=\sum_{\alpha}
\frac{(-q)_{\alpha}} {(-q)_{\sigma}}
{\det}_q(A_{\alpha_1 \ldots \alpha_t}^{\sigma_1 \ldots \sigma_t}){\det}_q(A_{\alpha_{t+1} \ldots \alpha_n}^{\sigma_{t+1}\ldots \sigma_n}),
\end{equation}
where the sum runs through all $t$-shuffles $\alpha\in S_n$ such that $\alpha_1<\cdots<\alpha_t, \alpha_{t+1}<\cdots<\alpha_n$.

In particular, for fixed $i, k$
\begin{equation}\label{e:Lap01}
\begin{aligned}
\delta_{ik}{\det}_q(A)&=\sum_{j=1}^n
\frac{\prod_{l<j}(-q_{lj})}{\prod_{l<i}(-q_{li})}a_{ij}
{\det}_q (A^{\hat{k}}_{\hat{j}})\\
&=\sum_{j=1}^n
\frac{\prod_{l>j}(-q_{jl})}{\prod_{l>i}(-q_{il})}
{\det}_q(A^{\hat{k}}_{\hat{j}})a_{ij},
\end{aligned}
\end{equation}
where $\hat{i}$ means the indices $1, \ldots, i-1, i+1, \ldots, n$ for brevity.

As for the quantum (column) determinant or column-minor, the corresponding Laplace expansion
for a fixed $r$-shuffle  $(\tau_1\ldots \tau_n)$ of $n$ such that $\tau_1<\cdots<\tau_r, \tau_{r+1}<\cdots<\tau_n$
\begin{equation}\label{e:lap2}
{\det}_q(A)
=\sum_{\beta}
\frac{(-p)_{\beta}} {(-p)_{\tau}}
{\det}_q(A^{\beta_1 \ldots \beta_r}_{\tau_1 \ldots \tau_r}){\det}_q(A^{\beta_{r+1} \ldots \beta_n}_{\tau_{r+1}\ldots \tau_n}),
\end{equation}
where the sum runs through all $r$-shuffles $\beta\in S_n$ such that $\beta_1<\cdots<\beta_r, \beta_{r+1}<\cdots<\beta_n$.

In particular, we have that for fixed $i, k$
\begin{align}\nonumber
\delta_{ik}{\det}_q(A)&=\sum_{j=1}^n
\frac{\prod_{l<j}(-p_{lj})}{\prod_{l<i}(-p_{li})}
a_{ji}{\det}_q(A_{\hat{k}}^{\hat{j}})\\ \label{e:Lap2}
&=\sum_{j=1}^n\frac{\prod_{l>j}(-p_{jl})}{\prod_{l>i}(-p_{il})}
{\det}_q(A^{\hat{j}}_{\hat{k}})a_{ji}.
\end{align}

\begin{theorem}\label{quasi-central}
In the bialgebra $\mathcal A $ one has that
$$a_{ij}{\det}_q(A)=\lambda^{j-i}\frac{\prod_{l=1}^{n}q_{li}}{\prod_{l=1}^{n}q_{lj}}{\det}_q(A) a_{ij}.$$
\end{theorem}
\begin{proof}
Let $X=(x_{ij}), A'=(a_{ij}'), A''=(a_{ij}'')$ be the matrices with entries in $\mathcal A $ defined by
\begin{align}
x_{ij}&={\det}_q (A^{\hat{j}}_{\hat{i}}),\\
a'_{ij}&=\frac{\prod_{l<j}(-q_{lj})}{\prod_{l<i}(-q_{li})}a_{ij},\\
a''_{ij}&=\frac{\prod_{l>i}(-p_{il})}{\prod_{l>j}(-p_{jl})}a_{ij}.
\end{align}
It follows from the Laplace expansion that
$$A' {\det}_q=A'XA''={\det}_q A''.$$
Therefore $\frac{\prod_{l<j}(-q_{lj})}{\prod_{l<i}(-q_{li})}a_{ij}{\det}_q(A)
=\frac{\prod_{l>i}(-p_{il})}{\prod_{l>j}(-p_{jl})}{\det}_q(A) a_{ij}.$  This is exactly

\begin{equation}\label{det commute}
a_{ij}{\det}_q(A)=\lambda^{j-i}\frac{\prod_{l=1}^{n}q_{li}}{\prod_{l=1}^{n}q_{lj}}{\det}_q(A) a_{ij}.
\end{equation}
\end{proof}

\begin{remark}\cite{BG}
 ${\det}_q(A)$ is central if and only if $\lambda^{j-i}\prod_{l=1}^{n}q_{li}=\prod_{l=1}^{n}q_{lj}$ for any $i,j$.
\end{remark}

Theorem \ref{quasi-central} implies that ${\det}_q(A)$ is a regular element in the bialgebra $\mathcal A $, therefore we can define the localization $\mathcal A [{{\det}_q}^{-1}]$, which will be denoted as $\mbox{GL}_{p, \lambda}(n)$. In fact, Theorem \ref{quasi-central} gives the following identity:

\begin{equation}\label{inverse of det}
{\det}_q(A)^{-1}a_{ij}=\lambda^{j-i}\frac{\prod_{l=1}^{n}q_{li}}{\prod_{l=1}^{n}q_{lj}} a_{ij}{\det}_q(A)^{-1}.
\end{equation}

By defining the antipode
\begin{equation}
\begin{aligned}
S(a_{ij})&=\frac{\prod_{l<i}-q_{li}}{\prod_{l<j}-q_{lj}}{\det}_q (A^{\hat{j}}_{\hat{i}}){\det}_q(A)^{-1}\\
&=\frac{\prod_{l>j}-p_{jl}}{\prod_{l>i}-p_{il}}{\det}_q(A)^{-1}{\det}_q(A^{\hat{j}}_{\hat{i}})
\end{aligned}
\end{equation}
the bialgebra $\mathcal A [{{\det}_q}^{-1}]$ becomes a Hopf algebra, thus a quantum group
in the sense of Drinfeld.

In fact, the second equation follows from (\ref{inverse of det}). Therefore, $AS(A)=S(A)A=I$ by the Laplace expansions. Subsequently
$$(\id\otimes S)\Delta=(S\otimes \id)\Delta=\varepsilon. $$

\section{Quasideterminants}

In this section we will work with the ring of fractions of noncommutative elements.
First of all let us recall some basic facts about quasideterminants.
  Let $X$ be the set of $n^2$ elements
$x_{ij}, 1\leq i,j\leq n$. For convenience, we also use $X$ to denote the matrix
$(x_{ij})$ over the ring generated by $x_{ij}$.

Denote by $F(X)$ the free division ring generated by $0,1,x_{ij}, 1\leq i,j\leq n$.
It is well-known that the matrix $X=(x_{ij})$ is an invertible element over $F(X)$
 \cite{GGRW} .


Let $I,J$ be two finite subsets of cardinality $k\leq n$ inside $\{1, \ldots, n\}$.
Following \cite{GGRW}, we introduce the notion of quasiderminant.
\begin{definition}
For $i\in I, j\in J$, the $(i,j)$-th quasideterminant $|X|_{ij}$ 
is the following element of $F(X)$:
$$|X|_{ij}=y_{ji}^{-1},$$
where $Y=X^{-1}=(y_{ij})$.
\end{definition}

If $n=1$, $I=i, J=j$. Then $|X|_{ij}=x_{ij}$.

When $n\geq 2$, and let $X^{ij}$ be the $(n-1)\times(n-1)$-matrix obtained from X by
deleting the $i$th row and $j$th column. In general $X^{i_1\cdots i_r, j_1\cdots j_r}$ denotes
the submatrix obtained from $X$ by deleting the $i_1, \cdots, i_r$-th rows, and
$i_1, \cdots, i_r$-th columns. Then
$$|X|_{ij}=x_{ij}-\sum_{i', j'} x_{ii'}(|X^{ij}|_{j'i'})x_{j'j},$$
where the sum runs over $i'\notin I\setminus\{i\}, j'\notin J\setminus\{j\}$.

\begin{theorem}

Let $A$ be the matrix of generators of $\mbox{GL}_{p, \lambda}(n)$.
In the ring of fractions of elements of $\mbox{GL}_{p, \lambda}(n)$, one has that
 \begin{equation}\label{det quasidet1}
 {\det}_q(A)=|A|_{11}|A^{11}|_{22}|A^{12,12}|_{33}\cdots a_{nn}
 \end{equation}
and the quasi-minors in the right-hand side commute with each other. More generally, for
 two permutations $\sigma$ 
and $\tau$ 
of $S_n$, one has that
 \begin{equation}\label{det quasidet2}
 {\det}_q(A)=\frac{(-q)_{\tau}}{(-q)_{\sigma}}|A|_{\si_1\tau_1}|A^{\si_1\tau_1}|_{\si_2\tau_2}|A^{\si_1\si_2,\tau_1\tau_2}|_{\si_3\tau_3}\cdots a_{\si_n\tau_n}.
 \end{equation}
\end{theorem}

\begin{proof}
By definition the quasi-determinants of $A$ are  inverses of the entries of the antipode $S(A)$,

\begin{equation}
\begin{aligned}
|A|_{ij}=S(a_{ji})^{-1}&=\frac{\prod_{l<i}(-q_{li})}{\prod_{l<j}(-q_{lj})}{\det}_q(A){\det}_q (A^{ij})^{-1}\\
&=\frac{\prod_{l>j}(-p_{jl})}{\prod_{l>i}(-p_{il})}{\det}_q(A^{ij})^{-1}{\det}_q(A),
\end{aligned}
\end{equation}
then
\begin{equation}
\begin{aligned}
\frac{\prod_{l<j}(-q_{lj})}{\prod_{l<i}(-q_{li})}|A|_{ij}{\det}_q (A^{ij})={\det}_q(A) \\
\frac{\prod_{l>i}(-p_{il})}{\prod_{l>j}(-p_{jl})} {\det}_q(A^{ij})|A|_{ij}= {\det}_q(A),
\end{aligned}
\end{equation}
By induction on the size of the matrix $A$, one sees that (\ref{det quasidet1}) and (\ref{det quasidet2}) hold.

It follows from \eqref{det commute} that ${\det}_q(A^{\{1,\ldots,s\}\{1,\ldots,s\}})$ and
 ${\det}_q(A^{\{1,\ldots,t\}\{1,\ldots,t\}})$ commute for $1\leq s,t\leq n-1$. Any factor on the right hand side of \eqref{det quasidet1} can be expressed as ${\det}_q (A^{\{1,\ldots,s\}\{1,\ldots,s\}}) {{\det}_q (A^{\{1,\ldots,s+1\}\{1,\ldots,s+1\}})}^{-1}$ multiplied by a scalar,
 therefore they commute with each other.
\end{proof}

\section{Multiparameter quantum Pfaffians}

\begin{definition}
Let $B=(b_{ij})$ be an $2n\times 2n$ square $p$-antisymmetric matrix with noncommutative entries such that
$b_{ji}=-p_{ij}b_{ij}, i<j$. The
multiparameter quantum $q$-Pfaffian is defined by

\begin{align*}
\Pf_q(B)
=\sum_{\sigma\in \Pi}(-q)_{\sigma}b_{\sigma(1)\sigma(2)}b_{\sigma(3)\sigma(4)}\cdots b_{\sigma(2n-1)\sigma(2n)},
\end{align*}
where $p=(p_{ij}),q=(q_{ij}),i<j$, and the sum runs through the set $\Pi$ of permutations $\sigma$ of $2n$ such that
$\sigma(2i-1)<\sigma(2i), i=1,\ldots,n.$
\end{definition}

Note that the parameters $q_{ij}$ and $p_{ij}$ satisfy the $(p,\lambda)$ condition: $p_{ij}q_{ij}=\lambda $.
\begin{proposition}\label{Pfaffian-Lap}
For any $0\leq t\leq n$,
\begin{equation}
\Pf_q(B)=\sum_{I} inv(I,I^c) \Pf_q(B_{I})\Pf_q(B_{I^c}),
\end{equation}
where the sum is taken over all subsets $I=\{i_1\cdots i_{2t}|i_1<\cdots <i_{2t}\}$ of $[1,2n]$, and
\begin{equation}
inv(I,J)=\prod_{i\in I, j\in J, i>j} (-q_{ji}).
\end{equation}
\end{proposition}
\begin{proof}
Let $\Omega=\sum_{i<j}b_{ij}x_{i}x_{j}$, where $x_i\in\Lambda_q(x)$. Then
\begin{equation}
\bigwedge^{n}\Omega=\Pf_q(B)x_{1}\wedge\cdots\wedge x_{2n}.
\end{equation}

On the other hand,

\begin{equation}
\begin{aligned}
\bigwedge^{n}\Omega&=\Omega^t\bigwedge\Omega^{n-t}\\
&=\sum_{I,J}\Pf(B_{I})x_I\Pf(B_{I^c})x_{J}\\
&=\sum_{I,J}\Pf(B_{I})\Pf(B_{I^c})x_Ix_{J}
\end{aligned}
\end{equation}
It is easy to see that $x_Ix_{J}$ vanishes unless $J=I^c$. Therefore

\begin{align*}
&\bigwedge^{n}\Omega
=\sum_{I}\Pf(B_{I})\Pf(B_{I^c})x_Ix_{I^c}\\
&=\sum_{I}inv(I,I^c)\Pf(B_{I})\Pf(B_{I^c})x_{1}\wedge\cdots\wedge x_{2n}.
\end{align*}
Thus we conclude that
$$\Pf(B)=\sum_{I}inv(I,I^c)\Pf(B_{I})\Pf(B_{I^c}).$$
\end{proof}

\begin{theorem}
Let $B=(b_{ij})_{1\leq i,j \leq 2n}$ be the $p$-antisymmetric matrix such that $b_{ji}=-p_{ij}b_{ij}, i<j$, and assume that
 the entries of $B$ commute with those of a $(p, \lambda)$-matrix $A=(a_{ij})_{1\leq i,j\leq 2n}$. Let $C=A^{T}BA$. Then
\begin{equation}
c_{ji}=-p_{ij}c_{ij}, \quad i<j¡£
\end{equation}
and
\begin{equation}
\Pf_{q}(C)={\det}_q(A)\Pf_{q}(B).
\end{equation}
\end{theorem}

\begin{proof} We first check that $c_{ij}$ also form anti-symmetric matrix. We compute that
\begin{align*}
c_{ii}&=\sum_{k,l}a_{ki}b_{kl}a_{li}=\sum_{k<l}a_{ki}b_{kl}a_{li}+a_{li}b_{lk}a_{ki}\\
&=\sum_{k<l}(a_{ki}a_{li}-p_{kl}a_{li}a_{ki})b_{kl}=0.\\
\end{align*}
For $i<j$,
\begin{align*}
c_{ij}&=\sum_{k,l}a_{ki}b_{kl}a_{lj}
=\sum_{k<l}\left(a_{ki}b_{kl}a_{lj}+a_{li}b_{lk}a_{kj}\right)\\
&=\sum_{k<l}(a_{ki}a_{lj}-p_{kl}a_{li}a_{kj})b_{kl}
=\sum_{k<l}{\det}_q(A^{kl}_{ij})b_{kl},\\
\end{align*}
\begin{align*}
c_{ji}&=\sum_{k,l}a_{kj}b_{kl}a_{li}=\sum_{k<l}(a_{kj}a_{li}-p_{kl}a_{lj}a_{ki})b_{kl}\\
&=\sum_{k<l}-p_{ij}(a_{ki}a_{lj}-p_{kl}a_{li}a_{kj})b_{kl}
=-p_{ij}\sum_{k<l}{\det}_q(A^{kl}_{ij})b_{kl}\\
&=-p_{ij}c_{ij}.\\
\end{align*}

Consider the element
$$\Omega=x^{t}Cx,$$
where we recall that $x=(x_1,\ldots,x_n)^t$ and $x_i\in \Lambda_q(x)$. Explicitly we have that
$\Omega=\sum_{1\leq i,j\leq n} c_{ij}x_{i}x_{j}=\sum_{i<j}(1+\lambda) c_{ij}x_{i}x_{j}$, therefore
\begin{equation}\label{1}
\bigwedge^{n}\Omega=(1+\lambda)^{n}\Pf_q(C)x_{1}\wedge\cdots\wedge x_{2n}.
\end{equation}

On the other hand, let $\omega_i=\sum_{j=1}^na_{ij}x_j$.
Then
\begin{align*}
\omega_j\omega_i&=-q_{ij}\omega_i\omega_j,\quad i<j, \\
\omega_i\omega_i&=0.
\end{align*}
As $Ax=(\omega_1,\ldots,\omega_{2n})^{t}$, one has that
\begin{align*}
\Omega&=x^{t}A^{t}BAx=(Ax)^{t}B(Ax)\\
&=\sum_{1\leq i,j\leq n}b_{ij}\omega_{i}\omega_{j}\\
&=\sum_{i<j}(1+\lambda)b_{ij}\omega_{i}\omega_{j}.
\end{align*}
Therefore
\begin{equation}\label{2}
\begin{aligned}
&\bigwedge^{n}\Omega=(1+\lambda)^{n}\Pf_q(B)\omega_1\wedge\cdots\wedge\omega_{2n}\\
&=(1+\lambda)^{n}\Pf_q(B){\det}_q(A)x_{1}\wedge\cdots \wedge x_{2n}
\end{aligned}
\end{equation}
Subsequently we have proved that
$$\Pf_{q}(C)={\det}_q(A)\Pf_{q}(B).$$
\end{proof}

The following column analog is clear.
\begin{remark}
Let $B$ be any matrix with entries $b_{ij},1\leq i,j \leq 2n$  commuting with $a_{ij}$ and $b_{ji}=-q_{ij}b_{ij},i<j$.
Let $C=ABA^{t}$. Then $c_{ji}=-q_{ij}c_{ij}, i<j$ and $\Pf_p(C)={\det}_q(A)\Pf_p(B)$.
\end{remark}

\section{Multiparameter quantum hyper-Pfaffians}
We now generalize the notion of the quantum multiparameter Pfaffian to the quantum hyper-Pfaffian.
A hypermatrix $A=(A_{i_1\cdots i_n})$ is an array of entries indexed by several indices, while a matrix is
indexed by two indices.
\begin{definition}
Let $B$ be a hypermatrix with noncommutative entries $b_{i_1\cdots i_{m}},1\leq i_k\leq mn,k=1,\ldots, m.$
Multiparameter quantum hyper-Pfaffian is defined by

\begin{align*}
\Pf_q(B)
=\sum_{\sigma\in \Pi}(-q)_{\sigma}b_{\sigma(1)\cdots\sigma(m)}\cdots b_{\sigma(m(n-1)+1)\cdots\sigma(mn)},
\end{align*}
Here $\Pi$ is the set of permutations $\sigma$ of $mn$ such that
$\sigma((k-1)m+1)<\sigma((k-1)m+2)<\sigma(km), k=1,\ldots,n.$
\end{definition}
Note that the multiparameter Pfaffian uses only the entries $b_{i_1\cdots i_m}$, where
$i_1<\cdots <i_m$.

Similar to Proposition {\ref{Pfaffian-Lap}}, one has the following result.
\begin{proposition}
For any $0\leq t\leq n$,
\begin{equation}
\Pf(B)=\sum_{I}inv(I,I^c)\Pf(B_{I})\Pf(B_{I^c}),
\end{equation}
where $I$ runs through subsets of $[1,mn]$ such that $|I|=mt$.
\end{proposition}
\begin{proof}

Let $\Omega=\sum_{i_1<\cdots <i_{m}}b_{i_1\cdots i_{m}}x_{i_1}\wedge\cdots \wedge x_{i_m}$,
then one has that
\begin{equation}\label{11}
\bigwedge^{n}\Omega=\Pf_q(B)x_{1}\wedge\cdots\wedge x_{2n}.
\end{equation}
and
\begin{equation}\label{22}
\bigwedge^{n}\Omega
=\Omega^t\bigwedge\Omega^{n-t}
=\sum_{I,J}\Pf(B_{I})\Pf(B_{I^c})x_Ix_{J},
\end{equation}
where as usual we have put $x_I=x_{i_1}\wedge\cdots\wedge x_{i_m}$.
Comparing \eqref{11} and \eqref{22}, one has the statement.
\end{proof}

\begin{theorem}
Let $B=(b_{i_1\cdots i_m })_{1\leq i,j \leq n}$ be any hypermatrix with noncommutative entries commuting  
with those of the matrix $A=(a_{ij})$. Let
$$c_{I}=\sum_{J}{\det}_q(A^J_I) b_J,$$
then $\Pf_{q}(C)={\det}_q(A)\Pf_{q}(B)$.
\end{theorem}

\begin{proof}
Let $\delta_i=\sum_{j=1}^{mn}a_{ij}x_{j}$, and consider the element $\Omega=\sum c_{I}x_{I}$. It is clear that
\begin{equation}\label{hyperpf1}\Omega^n=\Pf_{q}(C)x_1\wedge\cdots\wedge x_{2n}.
\end{equation}

On the other hand,
$\Omega=\sum b_{J}\delta_{J}$. Then
\begin{equation}\label{hyperpf2}
\Omega^n=\Pf(B)\delta_1\wedge\cdots\wedge \delta_{2n}=\Pf(B){\det}_q(A)x_1\wedge\cdots\wedge x_{2n}.
\end{equation}

Comparing (\ref{hyperpf1}) and (\ref{hyperpf2}) we conclude that
$$\Pf_{q}(C)={\det}_q(A)\Pf_{q}(B).$$
\end{proof}

\begin{remark}
The column-analog is also true. In fact, one has the following result.
Let $B=(b_{i_1\cdots i_m })_{1\leq i,j \leq n}$ be any hypermatrix with noncommutative entries commuting
with those of the matrix $A=(a_{ij})$. Let
$$c_{I}=\sum_{J}{\det}_q(A^I_J) b_J,$$
then $\Pf_{p}(C)={\det}_q(A)\Pf_{p}(B)$.

\end{remark}

\bigskip
\centerline{\bf Acknowledgments}
\medskip
The work is supported by National Natural Science Foundation of China (
11531004), Fapesp (2015/05927-0) and Humboldt foundation.
Jing 
acknowledges the support of
Max-Planck Institute for Mathematics in the Sciences, Leipzig.
Both authors also thank South China University of Technology for support during the work.

\vskip 0.1in

\bibliographystyle{amsalpha}

\end{document}